\documentclass[a4paper,12pt,reqno]{amsart}
\usepackage{latexsym}
\usepackage{amsmath,amsthm,amssymb,amscd}
\usepackage{fullpage}
\usepackage{xcolor}
\usepackage{parskip}
\usepackage{mathtools}
\mathtoolsset{showonlyrefs}

\usepackage[activate={true,nocompatibility},final,tracking=true,kerning=true,spacing=true]{microtype}
\usepackage[pdfencoding=unicode,pdftex,bookmarks=true,bookmarksnumbered]{hyperref}
\definecolor{citecolour}{rgb}{0.0, 0.0, 0.8}
\colorlet{linkcolour}{green!50!black}
\hypersetup{colorlinks,breaklinks,
		linkcolor=linkcolour,citecolor=citecolour,
		filecolor=linkcolour, urlcolor=linkcolour}
\usepackage{version}

\newtheorem{theorem}{Theorem}

\newtheorem{lemma}[theorem]{Lemma}
\newtheorem{proposition}[theorem]{Proposition}
\newtheorem{corollary}[theorem]{Corollary}
\theoremstyle{definition}
\newtheorem{definition}[theorem]{Definition}
\newtheorem{remark}[theorem]{Remark}

\numberwithin{theorem}{section}
\numberwithin{equation}{section}
\numberwithin{table}{section}

\usepackage[shortlabels]{enumitem}
\begingroup
    \makeatletter
    \@for\theoremstyle:=definition,remark,plain\do{%
        \expandafter\g@addto@macro\csname th@\theoremstyle\endcsname{%
            \addtolength\thm@preskip\parskip
            }%
        }
\DeclareMathOperator{\A}{\mathfrak{A}}

\DeclareMathOperator{\X}{\mathfrak{X}}
\DeclareMathOperator{\Y}{\mathfrak{Y}}

\DeclareMathOperator{\Zent}{\mathbf{Z}}

\DeclareMathOperator{\Out}{\mathbf{Out}}
\DeclareMathOperator{\Aut}{\mathbf{Aut}}
\DeclareMathOperator{\PSL}{PSL}
\DeclareMathOperator{\Sz}{Sz}

\DeclareMathOperator{\cent}{\mathbf{C}}
\DeclareMathOperator{\fitt}{\mathbf{F}}
\DeclareMathOperator{\fratt}{\mathbf{\Phi}}
\DeclareMathOperator{\Soc}{\mathbf{Soc}}

\newcommand{\R}{\mathbb{\R}}

\newcommand{\E}{\mathfrak{E}}

\includeversion{comment}
\renewcommand{\leq}{\leqslant}
\renewcommand{\geq}{\geqslant}

\begin{document}
\title{On residuals of finite groups}
\author[S. Aivazidis and 
T.\,W. M\"uller]{Stefanos Aivazidis$^{\dagger}$ and
Thomas M\"uller$^*$} 

\address{$^{\dagger*}$Einstein Institute of Mathematics, Edmond J. Safra Campus, The Hebrew University of Jerusalem, Givat Ram, Jerusalem 9190401, Israel.}

\address{$^*$School of Mathematical Sciences, Queen Mary
\& Westfield College, University of London,
Mile End Road, London E1 4NS, United Kingdom.}

% \thanks{$^\dagger$}



\begin{abstract}
A theorem of Dolfi, Herzog, Kaplan, and Lev \cite[Thm.~C]{DHKL} asserts that in a finite group with trivial Fitting subgroup, the size of the soluble residual of the group is bounded from below by a certain power of the group order, and that the inequality is sharp. Inspired by this result and some of the arguments in \cite{DHKL}, we establish the following generalisation: if $\X$ is a subgroup-closed Fitting formation of full characteristic which does not contain all finite groups and $\overline{\X}$ is the extension-closure of $\X$, then there exists an (optimal) constant $\gamma$ depending only on $\X$ such that,
for all non-trivial finite groups $G$ with trivial $\X$-radical,
\begin{equation}
\left\lvert G^{\overline{\X}}\right\rvert \,>\, \vert G\vert^\gamma,
\end{equation}
where $G^{\overline{\X}}$ is the ${\overline{\X}}$-residual of $G$. When $\X = \mathfrak{N}$, the class of finite nilpotent groups, it follows
that $\overline{\X} = \mathfrak{S}$, the class of finite soluble groups, thus we recover the original theorem of Dolfi, Herzog, Kaplan, and Lev. In the last section of our paper, building on J.\,G. Thompson's classification of minimal simple groups, we exhibit a family of subgroup-closed Fitting formations $\X$ of full characteristic such that $\mathfrak{S} \subset \overline{\X} \subset  \mathfrak{E}$, thus providing applications of our main result beyond the reach of \cite[Thm.~C]{DHKL}.
\end{abstract}

\maketitle

\section{Introduction}

In an important paper, dedicated to Avinoam Mann on the occasion of his retirement, Dolfi, Herzog, Kaplan, and Lev prove the following remarkable result.

\begin{theorem}[{\cite[Thm.~C]{DHKL}}]\label{Thm:SolRes}
Let $G$ be a non-trivial finite group such that $\fitt(G) =  1$. Then we have
\begin{equation}
\label{Eq:SolResEst}
\left\lvert G^\mathfrak{S}\right\rvert \,>\, \vert G\vert^\gamma,
\end{equation}
where 
\[\gamma = \log(60)/\log(120 (24)^{1/3}) \approx 0.700265861
\]
is largest possible for \eqref{Eq:SolResEst} to hold.
\end{theorem}

Here, $\mathfrak{S}$ denotes the class all finite soluble groups, $\fitt(G)$ is the Fitting subgroup of the finite group $G$, and $G^\mathfrak{S}$ is the soluble residual of $G$. The principal aim of the present paper is to clarify the (rather involved) background of Theorem~\ref{Thm:SolRes}, which relies on a delicate interplay between residuals and radicals, and to use this analysis for the purpose of establishing a substantial generalisation; see Theorem~\ref{Thm:Main} in Section~\ref{Sec:Main}. 

It is a well-known fact that, for a finite soluble group $G$, the Fitting subgroup $\fitt(G)$ controls triviality respectively non-triviality of $G$. 
As a closer analysis reveals, it is that property which makes the Fitting condition in Theorem~\ref{Thm:SolRes} natural and necessary. This observation indicates that our starting point for generalising Theorem~\ref{Thm:SolRes} should be the construction of a suitable extension $\overline{\X}$ of a given normal-product-closed class of finite groups $\X$, which is controlled by $\X$ in the sense that 
\begin{equation}\label{Eq:controlProp}
G\in \overline{\X} \,\mbox{ and }\, G_{\X} =  1 \,\mbox{ implies }\, G =  1
\end{equation}
(here, $G_{\X}$ denotes the $\X$-radical of the group $G$). It turns out that the natural candidate for $\overline{\X}$ is the class of all poly-$\X$ groups, and Section~\ref{Sec:ExtClos} takes a closer look at this construction of an ``extension-closure'' of a class of groups. 

% Section~\ref{Sec:OutEst} provides a short and elegant proof of the inequality 
% \[
% \left\lvert\Out(S)\right\rvert > \log_2 \vert S\vert,
% \]

% originally due to Kohl~\cite{SK}, which holds for all non-abelian finite simple groups. 

In Sections~\ref{Sec:Main} and \ref{Sec:Optimality} we prove our main result, Theorem~\ref{Thm:Main}. Sections~\ref{Sec:Res}--\ref{Sec:ClosInter} provide some useful general observations on residuals of finite groups, and their radicals, as well as recalling certain facts concerning the interdependence of closure properties of group classes. We include this material partly for the benefit of the reader, to make the paper more readable and self-contained, but also since most of the results mentioned in these three sections play a role in the proof of Theorem~\ref{Thm:Main}. Our final Section~\ref{Sec:Appl} discusses a class of applications of our main result, which are beyond the reach of Theorem~\ref{Thm:SolRes}.

\section{Residuals: some preliminaries}
\label{Sec:Res}
By a \emph{class of groups} $\X$ we shall mean a class $\X$ in the set-theoretic sense, whose members are finite groups, which contains the trivial group $1$, and is closed under isomorphism: if $G\in\X$ and $H\cong G$, then $H\in\X$. 
Let $\mathfrak{A}$, $\mathfrak{N}$, $\mathfrak{S}$, and $\mathfrak{E}$ denote, respectively, the class of finite abelian, finite nilpotent, finite soluble, and all finite groups. Also, as usual, we denote by $\Zent(G)$, $\fratt(G)$, and $\fitt(G)$, respectively, the centre, Frattini subgroup, and Fitting subgroup of the finite group $G$. We note that if $\X$ is a class of groups, $G$ a finite group, and $\alpha$ an automorphism of $G$, then $\alpha$ sends an $\X$-subgroup of $G$ to an $\X$-subgroup, and a normal $\X$-subgroup of $G$ to a normal $\X$-subgroup, since $\X$ is closed under isomorphisms. A class of groups $\X$ is called \emph{residually-closed}, or $\textsc{r}_0$-closed for short, if it satisfies
\[
G\in \mathfrak{E} \mbox{ and } N_1, N_2\unlhd G \mbox{ with }G\big/N_i\in\X \mbox{ for } i=1, 2 \mbox{ implies } G\big/(N_1\cap N_2)\in\X.
\]
Just as $\textsc{r}_0$-closure is associated with formations, so $\textsc{n}_0$-closure, its dual, is associated with Fitting classes. A class of groups $\X$ is called $\textsc{n}_0$-closed, if it satisfies
\[
G\in \mathfrak{E} \mbox{ and } N_1, N_2\unlhd\unlhd\, G \mbox{ with }N_i\in\X \mbox{ for } i=1, 2 \mbox{ and } G = \langle N_1, N_2\rangle\mbox{ implies } G\in\X.
\]

\begin{definition}
Given a residually-closed class of groups $\X$ and a finite group $G$, we define the $\X$-residual $G^{\X}$ of $G$ to be the smallest normal subgroup $N$ of $G$ such that $G/N\in\X$.  
\end{definition}

\begin{lemma}\label{Lem:Res}
Let $\X$ and $\Y$  be residually-closed classes of groups, and let $G$ be a finite group.
\begin{enumerate}[label={\upshape(\roman*)}]
\item The $\X$-residual $G^{\X}$ of $G$ exists, is a characteristic subgroup of $G,$  and is unique. Moreover, if $\X$ is also image-closed (that is, a formation), then
\begin{equation}
(G/N)^{\X} = G^{\X} N/N
\end{equation}
for each normal subgroup $N\unlhd G$.
\item We have $G^{\X} =  1$ if, and only if, $G\in\X$.
\item If $\X \subseteq \mathfrak{Y}$, then $G^{\mathfrak{Y}} \leq G^{\X}$. 
\end{enumerate}
\end{lemma}

\begin{proof}
(i) See \cite[II, Lem.~2.4]{DH}.
% Since $ 1\in\X$, the normal subgroup $N=G \unlhd G$ is such that $G/N\in\X$. Thus, the system of normal subgroups of $G$
% \[
% \mathcal{S} = \big\{N: N\unlhd G \mbox{ and } G/N\in\X\big\} = \big\{N_1, N_2, \ldots, N_s\big\}
% \]
% is non-empty and finite, as $G$ is finite. Set
% \[
% G^{\X}\coloneqq \bigcap \mathcal{S}. 
% \]
% As intersection of normal subgroups of $G$, $G^{\X}$ is itself a normal subgroup of $G$. Moreover, a straightforward induction on $\sigma$, $1\leq \sigma \leq s$, making use of the fact that $\X$ is residually-closed,  shows that $G/\bigcap_{\nu=1}^\sigma N_\nu\in \X$ for all $1\leq \sigma \leq s$. In particular, $G^{\X} = G/\bigcap_{\nu=1}^sN_\nu \in \mathcal{S}$, and is its least element under inclusion.

(ii) If $G^{\X} =  1$, then $G \cong G/G^{\X}\in\X$, so $G\in \X$. Conversely, if $G\in\X$, then the normal subgroup $N= 1 \unlhd G$ is such that $G/N \cong G \in\X$, thus $G/N\in\X$ and $G^{\X} \leq  1$, implying $G^{\X} =  1$.

(iii) If $N\unlhd G$ is such that $G/N\in \X$, then $G/N\in\mathfrak{Y}$ by hypothesis, thus $G^{\mathfrak{Y}} \leq N$. Taking $N=G^{\X}$, we find that $G^{\mathfrak{Y}} \leq G^{\X}$, as desired. 
\end{proof}

% As it turns out, there is a stable connection between the $\X$-residual of a group and that of its quotients, provided the residually-closed class of groups $\X$ is also image-closed.

% \begin{lemma}
% \label{Lem:QuoRes}
% Let $\X$ be a residually- and image-closed class of groups, and let $G$ be a finite group with normal subgroup $N\unlhd G$. Then we have 
% \begin{equation}
% (G/N)^{\X} = G^{\X} N/N.
% \end{equation}
% \end{lemma}

% \begin{proof}
% Since $G^{\X} \unlhd G$ by definition, we have $G^{\X} N/N \unlhd G/N$, and 
% \[
% (G/N)\big/(G^{\X} N/N) \, \cong\, G\big/(G^{\X} N) \,\cong\, (G/G^{\X})\big/((G^{\X} N)/N) \,\in\, \X
% \]
% since $G/G^{\X}\in \X$ by definition, and $\X$ is image-closed by hypothesis. Hence, $(G/N)^{\X} \leq G^{\X} N/N$. Now, let $H/N \unlhd G/N$ be any normal subgroup of $G/N$ such that 
% \[
% (G/N)\big/(H/N) \cong G/H\in\X. 
% \]
% Then $N \leq H\unlhd G$ and $G^{\X} \leq H$. Hence, $G^{\X} N \leq H$, so that $G^{\X} N/N \leq H/N$, showing that $G^{\X} N/N$ is the smallest normal subgroup of $G/N$ whose quotient lies in $\X$, whence the result.
% \end{proof}

% \begin{corollary}
% \label{Cor:QuoComm}
% Let $G$ be a finite group, and let $N\unlhd G$ be a normal subgroup. Then we have
% \[
% (G/N)' = G' N/N.
% \]
% \end{corollary}

% \begin{proof}
% The class $\X = \A$ of finite abelian groups is a formation; thus, in particular, residually- and image-closed. The result now follows from Lemma~\ref{Lem:Res} (i).
% \end{proof}
Concerning the abelian residuals of finite groups, the following was shown by Halasi and Podoski, and independently by Herzog, Kaplan, and Lev; cf. \cite[Thm.~1.1]{HP} and \cite[Thm.~A]{HKL}.

\begin{proposition}\label{Prop:AbRes}
Let $G$ be a finite group such that $\fratt(G) =  1$. Then we have
\begin{equation}
\label{Eq:AbResEst}
\left\lvert G'\right\rvert \,\geq\,\sqrt{(G : \Zent(G))},
\end{equation}
with equality occurring in~\eqref{Eq:AbResEst} if, and only if, $G$ is abelian.
\end{proposition}

Moreover, building on the work in~\cite{HKL1} and~\cite{HKL}, Guo and Gong have shown the following concerning the size of nilpotent residuals of finite groups; cf. \cite[Thm.~0.4]{GG}.

\begin{proposition}\label{Prop:NilRes}
Let $G$ be a finite group such that $\fratt(G) =  1$. Then the inequality
\begin{equation}
\label{Eq:NilResEst}
\left\lvert G'\right\rvert\cdot \left\lvert G^\mathfrak{N}\right\rvert \,\geq\, \left(G : \Zent(G)\right)
\end{equation}
holds true, with equality occurring in~\eqref{Eq:NilResEst} if, and only if, $G$ is abelian.
\end{proposition}
We observe that Propositions~\ref{Prop:AbRes} and ~\ref{Prop:NilRes} admit of a rather elegant common generalisation as follows.

\begin{proposition}
\label{Prop:AbNilGen}
Let $\X$ be a residually-closed class of groups such that $\A \subseteq \X \subseteq \mathfrak{N}$, and let $G$ be a finite group with $\fratt(G) = 1$. Then we have
\begin{equation}\label{Eq:AbNilGenEst}
\left(G : \Zent(G)\right)\, \leq \, \left\lvert G'\right\rvert \cdot \left\lvert G^{\X}\right\rvert,
\end{equation}
with equality occurring in \eqref{Eq:AbNilGenEst} if, and only if, $G$ is abelian.
\end{proposition}

\begin{proof}
This follows from Proposition~\ref{Prop:NilRes} in conjunction with Lemma~\ref{Lem:Res} (iii).
\end{proof}

\section{Radicals of finite groups}\label{Sec:Rad}
As the proof of our main result, Theorem~\ref{Thm:Main}, involves a delicate interplay between radical and residual theory, we shall use this section to briefly summarise what little can be said in general about radicals of finite groups. 

% The only observation possibly new here is Part~(iii) of Lemma~\ref{Lem:RadProps} below, which generalises a well-known property of the Fitting subgroup (with $\X$ replacing the formation $\mathfrak{N}$ of finite nilpotent groups); cf., for instance, ~\cite[Thm.~2.2]{isaacs}.

\begin{definition}
We say that a class of groups $\X$ is \emph{normal-product-closed}, if, for each finite group $G$, and any two normal subgroups $N_1, N_2\unlhd G$ with $N_1, N_2\in\X$, we have $N_1N_2\in \X$. 
\end{definition}

\begin{remark}\label{Rem:2gamma}
By \cite[II, Prop.~2.11(b)]{DH}, a class is normal-product-closed if, and only if, it is $\textsc{n}_0$-closed.
\end{remark}

\begin{lemma}\label{Lem:RadProps}
Let $\X$ be a normal-product-closed class of groups. 
\begin{enumerate}[label={\upshape(\roman*)}]
\item For each finite group $G$,  there exists a unique normal $\X$-subgroup $G_{\X}$ of $G$, which is largest in the sense that it contains every normal $\X$-subgroup of $G$. It is called the $\X$-radical of $G$, and is a characteristic subgroup of $G$.
\item If $G$ is a finite group, then we have $G_{\X} = G$ if, and only if, $G\in\X$.
\item Let $G$ be a finite group, and let $H$ be a subnormal $\X$-subgroup of $G$. Then $H\leq G_{\X}$.
\item Let $\Y$ be a normal-product-closed class of groups such that $\X\subseteq \Y$. Then, for each finite group $G$, we have $G_{\X} \leq G_{\Y}$.
\end{enumerate}
\end{lemma}

\begin{proof}
(i) See \cite[II, Lem.~2.9]{DH}, noting that by \cite[II, Prop.~2.11(b)]{DH} a class is normal-product-closed if, and only if, it is $\textsc{n}_0$-closed.

% Let $G$ be a finite group, and let $N_1 =  1, N_2, \ldots, N_r$ be a list of all the normal $\X$-subgroups of $G$. Then, by induction on $s$, one sees that the complex product  $N_1N_2\cdots N_s$ satisfies $N_1N_2\cdots N_s \in \X$ for all $1\leq s\leq r$, since $\X$ is normal-product-closed. Hence, $G_{\X} = N_1 N_2 \cdots N_r$ is a normal $\X$-subgroup of $G$ containing every normal $\X$-subgroup of $G$, and is clearly unique with this property. Now, let $\alpha$ be an automorphism of $G$. Then $\alpha(G_{\X})$ is again a normal $\X$-subgroup of $G$, thus $\alpha(G_{\X}) \leq G_{\X}$, implying $\alpha(G_{\X}) = G_{\X}$, since $\vert \alpha(G_{\X})\vert = \vert G_{\X}\vert$. Hence, $G_{\X}$ is characteristic in $G$, as claimed.

(ii) This is trivial.

(iii) We argue by induction on $\vert G\vert$. If $G =  1$, then $H = G_{\X} = G$. Suppose now that our claim holds for all groups with $\vert G\vert <m$ for some integer $m\geq2$, and consider a group $G$ with $\vert G\vert = m$. If $H = G$, then $G\in\X$, so $H = G = G_{\X}$. We may therefore assume that $H < G$. Let $M$ be the penultimate term in a subnormal chain reaching from $H$ to $G$. Then $H$ is subnormal in $M$, $M$ is normal in $G$, and $M<G$. The induction hypothesis applied to $M$ shows that $H\leq M_{\X}$. However, the radical $M_{\X}$ is characteristic in $M$, and $M$ is normal in $G$, thus $M_{\X}$ is a normal $\X$-subgroup of $G$. It follows that $M_{\X} \leq G_{\X}$, so that $H \leq G_{\X}$, as desired.

(iv) Let $G$ be a finite group. By Part~(i), the radicals $G_{\X}$ and $G_{\Y}$ exist. Moreover, since $\X \subseteq \Y$, the radical $G_{\X}$ is a normal $\Y$-subgroup of $G$, hence $G_{\X} \leq G_{\Y}$, as claimed.
\end{proof}
\section{Some interdependences between closure properties of group classes}
\label{Sec:ClosInter}
Our next result summarises some interdependences between closure properties of group classes, which will be useful later on. 

\begin{lemma}
\label{Lem:XProps}
Let $\X$ be a class of groups.
\begin{enumerate}[label={\upshape(\roman*)}]
\item If $\X$ is normal-product-closed, then $\X$ is closed under taking direct products.
\item If $\X$ is subgroup-closed, as well as closed under taking direct products, then $\X$ is residually-closed.
\item If $\X$ is image-closed as well as extension-closed, then $\X$ is normal-product-closed.
\item Suppose that $\X$ is image-closed, as well as normal-product-closed, let $G$ be a finite group, and let $N$ be a normal subgroup of $G$. Then we have 
\begin{equation}\label{Eq:QuoRad}
G_{\X} N/N \,\leq\, (G/N)_{\X}. 
\end{equation}
If $N\in \X$, and $\X$ is image-closed as well as extension-closed, then \eqref{Eq:QuoRad} holds with equality.
\item Suppose that $\X$ is both normal-product-closed as well as extension-closed, and let $G$ be an arbitrary finite group. Then $\left(G/G_{\X}\right)_{\X} =  1$.
\end{enumerate}
\end{lemma}
\begin{proof}
(i) Let $G_1, G_2\in\X$, and set $H\coloneqq G_1\times G_2$, the (external) direct product of $G_1$ and $G_2$. Then $H$ is a finite group and $\widetilde{G}_1\coloneqq \{(x, 1): x \in G_1\}$, $\widetilde{G}_2\coloneqq\{(1, y): y\in G_2\}$ are normal subgroups of $H$. Moreover, we have $\widetilde{G}_i\cong G_i$ for $i=1, 2$, thus $\widetilde{G}_i\in\X$, by assumption, and $\widetilde{G}_1 \widetilde{G}_2 = H$. Since $\X$ is normal-product-closed by hypothesis, it follows that $H\in\X$, as claimed.

(ii) See \cite[II, Lem.~1.18]{DH}.

% Let $G$ be a finite group, and let $M, N\unlhd G$ be normal subgroups of $G$ such that $G/M, G/N\in\X$. Then the quotient $G\big/(M\cap N)$ embeds homomorphically into the direct product $G/M\times G/N$ via the map $x(M\cap N) \overset{\varphi}{\longmapsto} (xM, xN)$; cf., for instance, Hilfssatz~9.6 in \cite[Chap.~I]{Huppert}. Since $G/M, G/N\in \X$ by hypothesis, and $\X$ is closed under taking direct products, we have 
% \[
% G/M \,\times\, G/N\,\in\,\X.
% \]
% Further, as $\X$ is assumed to be subgroup-closed, and
% \[
% G\big/(M\cap N) \,\cong\, \mathrm{Im}(\varphi) \,\leq\, G/M \times G/N,
% \]
% we conclude that $G\big/(M\cap N) \in \X$. Hence, $\X$ is residually-closed, as claimed.

(iii) Let $G$ be an arbitrary finite group, and let $N_1, N_2 \unlhd G$ be such that $N_1, N_2\in\X$. Then the normal product $N_1N_2$ in $G$ is an extension of, say, $N_1\in \X$ by
\[
N_1N_2/N_1 \,\cong\, N_2\big/(N_1\cap N_2)\,\in\,\X,
\]
since $N_2\in\X$, and $\X$ is image-closed. It follows that $N_1N_2\in\X$, since $\X$ is assumed to be extension-closed. Hence, $\X$ is normal-product-closed, as claimed.

(iv) Since $N \leq G_{\X} N \unlhd G$, $G_{\X} N/N$ is a normal  subgroup of $G/N$, and we have 
\[
G_{\X} N/N \,\cong\, G_{\X}\big/(G_{\X} \cap N) \in \X,
\]
since $G_{\X}\in \X$ by definition, and $\X$ is image-closed by hypothesis. The desired inclusion, $G_{\X} N/N \leq (G/N)_{\X}$, follows now from the definition of the $\X$-radical. 

Now, suppose that $N\in\X$, and that $\X$ is image-closed, as well as extension-closed. Then \eqref{Eq:QuoRad} holds by what we have just proven plus Part~(iii) of our lemma. Let $H/N$ be a normal subgroup of $G/N$ with $H/N\in \X$. Then $N\leq H\unlhd G$, so that $H$ is an extension of $N\in\X$ by $H/N\in \X$. Since $\X$ is extension-closed by assumption, we have $H\in \X$, thus $H\leq G_{\X}$ by definition of the $\X$-radical, implying $H/N \leq G_{\X} N/N$. Taking $H/N = (G/N)_{\X}$, it follows that $G_{\X} N/N = (G/N)_{\X}$, as claimed.

(v) Let $H = H_1/G_{\X}$ be an arbitrary normal $\X$-subgroup of $G/G_{\X}$, where $G_{\X} \leq H_1\unlhd G$. Then $H_1$ is an extension of the $\X$-subgroup $G_{\X} \unlhd H_1$ by the $\X$-group $H$; thus  $H_1$  is a normal $\X$-subgroup of $G$, since $\X$ is assumed to be extension-closed. It follows that $G_{\X} \leq H_1\leq G_{\X}$, so $H_1 = G_{\X}$ and thus $H =  1$, implying $(G/G_{\X})_{\X} =  1$ as desired, since $H$ was arbitrary.
\end{proof}

\section{Extension-closure of a class of groups}\label{Sec:ExtClos}
Let $\X$ be a class of (finite or infinite) groups. We define a class of groups $\overline{\X}$ via
\begin{multline}
\label{Eq:EClosure}
G\in\overline{\X} :\Longleftrightarrow \mbox{there exists a finite subnormal series 
$ 1 = G_0 \unlhd G_1 \unlhd \cdots \unlhd G_r = G$}\\
\mbox{such that $G_i/G_{i-1} \in \X$ for $1\leq i\leq r$.}
\end{multline}
\begin{remark}
If we interpret the class $\X$ as a group property, then the elements of $\overline{\X}$ are precisely the poly-$\X$ groups.
\end{remark}
\begin{lemma}
\label{Lem:CO}
The class map ${}^-$ sending $\X$ to $\overline{\X}$ is a closure operation in the sense of \cite[II, Def.~(1.4)~(a)]{DH}.
\end{lemma}

\begin{proof}
(1) Let $\X$ be a class of groups, and let $G\in \X$. We have a subnormal series
\[
 1 = G_0 \unlhd G_1 = G,
\]
where $G_1/G_0 \cong G\in \X$. Thus, $G\in\overline{\X}$, showing that the map ${}^-$ is expanding.

(2) Let $G\in\overline{\overline{\X}}$. Then there exists a finite subnormal series
\begin{equation}
\label{Eq:Sub1}
 1 = G_0 \unlhd G_1 \unlhd \cdots \unlhd G_r = G,
\end{equation}
where $G_i/G_{i-1} \in \overline{\X}$ for $1\leq i\leq r$. Moreover, by definition of $\overline{\X}$, we have subnormal series
\begin{equation}
\label{Eq:Sub2}
 1 = G_0^{(i)} \unlhd G_1^{(i)} \unlhd \cdots \unlhd G_{s_i}^{(i)} = G_i/G_{i-1},\quad 1\leq i\leq r,
\end{equation}
such that $G^{(i)}_{j_i}/G_{j_i-1}^{(i)} \in \X$ for $1\leq j_i\leq s_i$. Lifting \eqref{Eq:Sub2}, we obtain, for each $i$, a corresponding series
\begin{equation}
\label{Eq:Sub3}
G_{i-1} = H^{(i)}_0 \unlhd H_1^{(i)} \unlhd \cdots \unlhd H^{(i)}_{s_i} = G_i,
\end{equation}
where $H^{(i)}_{j_i}/G_{i-1} = G^{(i)}_{j_i}$. Also, we have 
\[
H^{(i)}_{j_i}/H^{(i)}_{j_i-1} \,\cong\, (H^{(i)}_{j_i}/G_{i-1})\big/(H^{(i)}_{j_i-1}/G_{i-1}) \,=\, G^{(i)}_{j_i}/G^{(i)}_{j_i-1}\,\in \,\X
\]
by construction. Interpolating the series \eqref{Eq:Sub3} in \eqref{Eq:Sub1}, we obtain a subnormal series reaching from the trivial group to the group $G$, whose quotients are all in $\X$. This shows that $\overline{\overline{\X}} \subseteq \overline{\X}$, while the reverse inclusion holds by the first part of the proof. Hence, the class map ${}^-$ is idempotent.

(3) Let $\X_1$ and $\X_2$ be classes of groups such that $\X_1 \subseteq \X_2$, and let $G\in\overline{\X}_1$. Then there exists a subnormal series 
\[
 1 = G_0 \unlhd G_1 \unlhd \cdots \unlhd G_r = G,
\] 
where $G_i/G_{i-1} \in \X_1$ for $1\leq i\leq r$. Since $\X_1\subseteq \X_2$, we also have $G_i/G_{i-1} \in \X_2$ for $1\leq i\leq r$, so that $G\in\overline{\X}_2$. Hence $\overline{\X}_1 \subseteq \overline{\X}_2$, showing that the map ${}^-$ is monotone as well.
\end{proof}

\begin{definition}
Given a class of groups $\X$, the class $\overline{\X}$ defined in \eqref{Eq:EClosure} is called the extension-closure (\textsc{e}-closure for short) of $\X$.
\end{definition}

Our next result collects together a number of properties of the extension closure $\overline{\X}$ of a class of groups $\X$, most of which will be crucial in what follows.

\begin{proposition}\label{Prop:EClosProps}
Let $\X$ be a class of groups, and let $\overline{\X}$ be its \textsc{e}-closure. 
\begin{enumerate}[label={\upshape(\roman*)}]
\item If the members of $\X$ are all finite groups, then the same holds true for $\overline{\X}$.
\item If the members of $\X$ all satisfy the maximum condition for subgroups, then the same holds for the members of $\overline{\X}$.
\item $\overline{\X}$ is closed under taking extensions.
\item If $\X$ is normal-product-closed, then $G\in\overline{\X}\cap \mathfrak{E}$ and $G_{\X}= 1$ implies $G= 1$.
\item Let $\X$ be normal-product-closed, and suppose that $\X$ consists solely of finite groups. Moreover, suppose that there exists a finite group $G$ such that $G\neq 1$ and $G_{\X}= 1$. Then $\overline{\X} \subset \mathfrak{E}$.
\item If $\X$ is image-closed, then so is $\overline{\X}$.
\item If $\X$ is subgroup-closed, then so is $\overline{\X}$.
\item If $\X$ is image-closed and subgroup-closed, then $\overline{\X}$ is normal-product-closed as well as  residually-closed.
\item Suppose that $\X$ is image-closed and normal-product-closed. Then, for $G\in\mathfrak{E}$, we have
\begin{multline*}
G \in\overline{\X}\,\Longleftrightarrow\, \mbox{there exists a series $ 1 = G_0 \leq G_1 \leq \cdots \leq G_r = G$, 
with each}\\
\mbox{subgroup $G_j$ characteristic in $G$, and $G_i/G_{i-1}\in\X$ for $1\leq i\leq r$.}
\end{multline*}
\end{enumerate}
\end{proposition}

\begin{proof}
(i) Let $G\in\overline{\X}$, and let $ 1 =G_0 \unlhd G_1\unlhd \cdots \unlhd G_r = G$ be a subnormal series for $G$ with $G_i/G_{i-1}\in\X$ for $1\leq i\leq r$. If $G_0, G_1, \ldots, G_{i-1}$ are finite for some $i$ with $1\leq i< r$, then $G_i$ is an extension of the finite group $G_{i-1}$ by the group $G_i/G_{i-1}$, which is in $\X$ and thus finite, hence $G_i$ is finite as well. Since $G_0$ is finite, it follows that $G=G_r$ is finite, as claimed.

(ii) This follows from Proposition~1 in \cite[Chap.~1]{Segal} by an immediate induction on the length of a subnormal $\X$-series.

(iii) Let $N, Q\in\overline{\X}$, and let $G$ be an arbitrary extension of $N$ by $Q$; that is, $N\unlhd G$ and $G/N=Q$. We want to show that $G\in\overline{\X}$. By definition of $\overline{\X}$, we have subnormal series
\begin{equation}
\label{Eq:ExtClosed1}
 1 = N_0 \unlhd N_1 \unlhd \cdots \unlhd N_r = N
\end{equation}
and
\[
 1 = Q_0 \unlhd Q_1 \unlhd \cdots \unlhd Q_s = Q,
\]
where $N_i/N_{i-1} \in\X$ and $Q_j/Q_{j-1}\in\X$ for $1\leq i\leq r$ and $1\leq j\leq s$. For each index $j$, let $H_j$ be such that $N \leq H_j \leq G$ and $H_j/N = Q_j$. In this way, we obtain a subnormal series
\begin{equation}
\label{Eq:ExtClosed2}
N = H_0 \unlhd H_1 \unlhd \cdots \unlhd H_s = G,
\end{equation}
such that 
\[
H_j/H_{j-1} \,\cong\, (H_j/N)\big/(H_{j-1}/N) \,=\, Q_j/Q_{j-1}\,\in\,\X,\quad 1\leq j\leq s.
\]
Combining the two series \eqref{Eq:ExtClosed1} and \eqref{Eq:ExtClosed2}, we obtain a subnormal series reaching from the trivial group to $G$, and with all quotients in $\X$. Hence, $G\in\overline{\X}$, as required.

(iv) Let $G\in \overline{\X}$ be a finite group such that $G_{\X} = 1$, and let
\[
 1 = G_0 \unlhd G_1 \unlhd \cdots \unlhd G_r = G
\]
be a subnormal series for $G$ with $G_i/G_{i-1}\in\X$ for $1\leq i\leq r$. Then each subgroup $G_i$ is subnormal in $G$, and if $G_{i-1} =  1$, then $G_i \cong G_i/G_{i-1}\in\X$, thus $G_i\leq G_{\X}$ by Part~(iii) of Lemma~\ref{Lem:RadProps}. A trivial induction on $i$ now shows that $G_i= 1$ for all $i=0,1\ldots, r$; in particular, $G= G_r = 1$, as claimed.

(v) Let $G$ be a non-trivial finite group such that $G_{\X} = 1$. If we had $G\in\overline{\X}$ then, by Part~(iv), we would have $G=1$, a contradiction. Hence, $G\not\in\overline{\X}$ while, by Part~(i), $\overline{\X} \subseteq \mathfrak{E}$. Hence, $\overline{\X} \subset \mathfrak{E}$, as claimed. 

(vi) Suppose that $\X$ is image-closed, let $G\in\overline{\X}$, and let $N\unlhd G$. We want to show that $G/N\in\overline{\X}$. Let
\[
 1 = G_0 \unlhd G_1 \unlhd \cdots \unlhd G_r =G
\]
be a subnormal series for $G$ such that $G_i/G_{i-1}\in \X$ for $1\leq i\leq r$. Set 
\[
\overline{G}_j\coloneqq G_j N/N,\quad 0\leq j\leq r.
\]
Then $\overline{G}_j \leq G/N$, $\overline{G}_0 =  1$, $\overline{G}_r = G/N$, and $\overline{G}_{i-1} \unlhd \overline{G}_i$ for $1\leq i\leq r$. Only the last assertion may need some comment. Clearly, $\overline{G}_{i-1} \leq \overline{G}_i$. Let $x\in G_{i-1}$, $x'\in G_i$, and let $y, y'\in N$. Then, since $N$ is normal in $G$, there exists some $y''\in N$, such that
\[
(xy)^{x'y'}\, =\, (y')^{-1} (x')^{-1} x y x' y' \,=\, (x')^{-1} x x' y'' \,\in \, G_{i-1} N,
\]
since $G_{i-1} \unlhd G_i$. It follows that $G_{i-1}N \unlhd G_i N$, and so $\overline{G}_{i-1} \unlhd \overline{G}_i$ for $1\leq i\leq r$. Hence, we obtain a subnormal series
\[
 1 = \overline{G}_0 \unlhd \overline{G}_1 \unlhd \cdots \unlhd \overline{G}_r = G/N
\]
for the quotient group $G/N$. Also, for $1\leq i \leq r$, 
\begin{multline*}
\overline{G}_i/\overline{G}_{i-1} \,=\, (G_i N/N)\big/(G_{i-1}N/N) \,\cong\,G_iN/G_{i-1}N\,=\, G_i \cdot G_{i-1}N/ G_{i-1}N \\
\cong\,G_i\big/(G_i \cap G_{i-1}N) \,\cong\, (G_i/G_{i-1})\big/((G_i \cap G_{i-1}N)/G_{i-1}) \,\in\, \X,
\end{multline*}
since $G_i/G_{i-1} \in\X$ and $\X$ is image-closed by hypothesis. Hence, $G/N\in\overline{\X}$, as desired.

(vii) Suppose that $\X$ is subgroup-closed, let $G\in \overline{\X}$, and let $U\leq G$ be an arbitrary subgroup. We want to show that $U\in\overline{\X}$. Fix a subnormal series
\[
 1 = G_0 \unlhd G_1 \unlhd \cdots \unlhd G_r = G
\]
with $G_i/G_{i-1}\in\X$ for $1\leq i\leq r$. Setting
\[
U_j\coloneqq G_j\cap U,\quad 0\leq j\leq r,
\]
we obtain a corresponding series
\begin{equation}
\label{Eq:SubGpClosed}
 1 = U_0 \leq U_1 \leq \cdots \leq U_r = U
\end{equation}
for $U$. Let $x\in U_{i-1}$ and $y\in U_i$. Then $x^y\in G_{i-1}$ since $G_{i-1} \unlhd G_i$, and $x^y\in U$, since $x, y\in U$. Thus, $x^y\in U_{i-1}$, so $U_{i-1} \unlhd U_i$ for $1\leq i\leq r$. Hence, \eqref{Eq:SubGpClosed} is a subnormal series for $U$. Moreover, the quotient $U_i/U_{i-1}$ embeds homomorphically into $G_i/G_{i-1}$ via the map sending $x U_{i-1}$ to $x G_{i-1}$ for $x\in U_i$. Since $G_i/G_{i-1}\in\X$ and $\X$ is subgroup-closed by hypothesis, we have $U_i/U_{i-1} \in\X$ for $1\leq i\leq r$, implying that $U\in\overline{\X}$, as claimed.

(viii) This follows from Parts~(iii), (vi), and (vii) of our proposition in conjunction with Parts (i)--(iii) of Lemma~\ref{Lem:XProps}.

(ix) The backward implication is clear by definition of the class $\overline{\X}$, so we may focus on the forward implication. By our hypotheses on $\X$ plus Parts~(iv) and (vi) of our proposition, the class $\overline{\X}$ is image-closed, and satisfies the control property \eqref{Eq:controlProp} for all finite groups $G$. 
Let $G\in\overline{\X} \cap \mathfrak{E}$, and suppose that $G \neq  1$ (otherwise there is nothing to prove). Then $G_{\X} \neq  1$ by \eqref{Eq:controlProp}, and we may set $G_0\coloneqq  1$ and $G_1\coloneqq G_{\X}$, so that $G_1/G_0 \cong G_{\X}\in\X$ and $G_1$ is characteristic in $G$ by Part~(i) of Lemma~\ref{Lem:RadProps}. Suppose that we have already constructed subgroups $G_0$, $G_1, \cdots G_{j-1}$, all characteristic in $G$, such that $G_1\leq \cdots\leq G_{j-1}$ and such that $G_i/G_{i-1}\in\X$ for $1\leq i<j$ with some integer $j\geq 2$. If $G_{j-1} = G$, we are done. Otherwise, $G/G_{j-1} \neq  1$, and $G/G_{j-1}\in\overline{\X}\cap \mathfrak{E}$, since $G\in\overline{\X} \cap \mathfrak{E}$ and $\overline{\X}$ is image-closed. By \eqref{Eq:controlProp}, we have $(G/G_{j-1})_{\X} \neq  1$, and we let $G_j$ be such that $G_{j-1} < G_j \leq G$ and $G_j/G_{j-1} = (G/G_{j-1})_{\X} \in\X$. Moreover, since $(G/G_{j-1})_{\X}$ is characteristic in $G/G_{j-1}$, again by Part~(i) of Lemma~\ref{Lem:RadProps}, the subgroup $G_j$ is characteristic in $G$. Continuing in this way, we find, after finitely many steps, a series $ 1 = G_0 \leq G_1 \leq \cdots \leq G_r = G$ for $G$, such that $G_1, \ldots, G_{r-1}$ are characteristic in $G$, and with $G_i/G_{i-1}\in\X$ for $1\leq i\leq r$.
\end{proof}

\begin{corollary}
\label{Cor:ExtSuperProps}
Let $\X$ be a class of finite groups, which is image-closed, subgroup-closed, and normal-product-closed. Suppose further that there exists a finite group $G$ with $G\neq  1$ and $G_{\X}= 1$. Then $\overline{\X}$ is image-closed, subgroup-closed, extension-closed, normal-product-closed, and residually-closed. Also, $G\in\overline{\X}$ and $G_{\X}= 1$ implies $G= 1$, and we have $\overline{\X} \subset \mathfrak{E}$.
\end{corollary}
% \section{An estimate for the number of outer automorphisms of a finite non-abelian simple group}
% \label{Sec:OutEst}
% The following estimate for the size of the outer automorphism group of a non-abelian simple group in terms of its order is due to Kohl~\cite{SK}. Below we present our own proof, which is considerably shorter and more elegant.
% \begin{lemma}
% \label{Lem:KohlEst}
% Let $\mathfrak{J}$ denote the class of all finite simple groups and, as before, let $\A$ be the class of finite abelian groups. Then we have 
% \begin{equation}
% \label{Eq:KohlEst}
% \left\lvert \Out(S)\right\rvert\,>\, \log_2 \vert S\vert,\quad S\in\mathfrak{J}\setminus\A.
% \end{equation}
% \end{lemma}
% \begin{proof}
% {\bf Fill in proof}
% \end{proof}
\section{The main result}
\label{Sec:Main}
We begin by recasting our hypotheses in the language of formations and Fitting classes.
\begin{lemma}\label{Lem:HypTranslate}
If $\X$ is a class of finite groups, then the following assertions are equivalent:
\begin{enumerate}[label={\upshape(\roman*)}]
\item $\X$ is subgroup-closed, image-closed, and normal-product-closed;
\item $\X$ is a subgroup-closed Fitting formation.
\end{enumerate}
\end{lemma}
\begin{proof}
Since $\X$ is subgroup-closed and normal-product-closed, it is $\textsc{s}_{n}$-closed (subnormal-subgroup-closed), and also $\textsc{n}_0$-closed by the backward implication in \cite[II, Prop.~2.11~(b)]{DH}; cf. \cite[II, (1.5)]{DH} for the closure operations mentioned here. Hence, $\X$ is a Fitting class. Also, since normal-product-closure for a class of groups implies direct-product-closure (see Part~(i) of Lemma~\ref{Lem:XProps}), and direct-product-closure plus subgroup-closure implies $\textsc{r}_0$-closure by \cite[II, Lem.~1.18]{DH}, $\X$ is a formation, thus a subgroup-closed Fitting formation. Conversely, if $\X$ is a subgroup-closed Fitting formation, then $\X$ is obviously subgroup-closed and image-closed, and is normal-product-closed by the forward implication in \cite[II, Prop.~2.11(b)]{DH}.
\end{proof}

\begin{definition}
Let $\X$ be a class of groups. The \emph{characteristic} of $\X$ is defined as 
\[
\mathrm{char}(\X) \coloneqq \left\{ p: p \in \mathbb{P} \mbox{ and } C_p \in \X \right\}, 
\]
where $\mathbb{P}$ denotes the set of prime numbers, and we say that $\X$ is of \emph{full characteristic} if $\mathbb{P} = \mathrm{char}(\X)$.
\end{definition}

Later, we shall need to assume that our class $\X$ is not only a subgroup-closed Fitting formation, but also that $\mathfrak{A} \subseteq \X$. For that to hold, however, it suffices to assume that $\X$ is of full characteristic, in which case we have the stronger result that $\mathfrak{N} \subseteq \X$.
% Evidently, if $\X$ is a class of soluble groups of full characteristic, then $\overline{\X} = \mathfrak{S}$. And if $\X$ is, in addition, a subgroup-closed Fitting formation, then $\mathfrak{N} \subseteq \X$. 
\begin{lemma}\label{Lem:brycecossey}
If $\X$ is a subgroup-closed Fitting formation of full characteristic, then $\mathfrak{N} \subseteq \X$.
\end{lemma}
\begin{proof}
Let $\X$ be a subgroup-closed Fitting formation of full characteristic, write $\X_{\text{sol}} = \X \cap \mathfrak{S}$, and note that $\X_{\text{sol}}$ is also a subgroup-closed Fitting formation of full characteristic. Now, fix a prime $p$. We have $C_p \in \X_{\text{sol}}$, thus $\X_{\text{sol}}$, which is closed under taking direct products, contains all elementary abelian $p$-groups. A theorem of Bryce and Cossey asserts that every subgroup-closed Fitting formation of soluble groups is saturated; cf. \cite[XI, Thm.~1.1]{DH}, or the original \cite[Thm.~1]{fittform}. It follows  that $P/\fratt(P) \in \X_{\text{sol}}$ for all $p$-groups $P$. But $\X_{\text{sol}}$ is saturated, so $P \in \X_{\text{sol}}$, and hence $\X_{\text{sol}}$ contains all $p$-groups. Since $\X_{\text{sol}}$ is of full characteristic, this is true for all primes $p$. The direct-product-closure of $\X_{\text{sol}}$ now yields $\mathfrak{N} \subseteq \X_{\text{sol}} \subseteq \X$, as claimed. 
\end{proof}
% Therefore, if $\X \subseteq \mathfrak{S}$ is a subgroup-closed Fitting formation of full characteristic, then $\mathfrak{N} \subseteq \X \subseteq \mathfrak{S} = \overline{\X}$, as wanted.
\subsection{Prolegomena to Theorem \ref{Thm:Main}}\label{prolegomena}

Before we carry on with the proof of our main result, we need to briefly discuss \cite[Thm.~5.8A]{DM}. This theorem states that if $\mathfrak{Y}$ (later, we will take $\mathfrak{Y} = \overline{\X}$ with $\X$ a subgroup-closed Fitting formation) is a non-empty class of finite groups, which is image-closed and subgroup-closed, and does not contain every finite group, then there exists an absolute constant $c>1$, which depends on $\mathfrak{Y}$ (we will write this as $\beta_{\overline{\X}}$ in the proof of our main result below), such that for all $n \geq 1{:}$
\[
\mbox{if }\, G \leq S_n \mbox{ and } \, G \in \mathfrak{Y},  \mbox{ then } \, |G| \leq c^{n-1} \tag*{($\dagger$)}
\]
(note in this context that $\X\subset \E$ implies $\overline{\X} \subset \E$ by Corollary \ref{Cor:ExtSuperProps} and Lemma \ref{Lem:HypTranslate}).

In order to define this constant $c$ (relative to the class $\mathfrak{Y}$), one needs to define first an auxiliary constant $c_0$.

Since $\mathfrak{Y}$ does not contain every finite group, there exists a smallest positive integer $m_0$ such that $A_{m_0} \notin \mathfrak{Y}$, since every finite group can be embedded as a subgroup into some alternating group ($G \leq A_{|G|+2})$. Then $c_0$ is taken to be the minimal constant such that $(\dagger)$ holds with $c=c_0$ for all natural numbers $n \leq m_0$. In fact, if $\mathfrak{Y}$ is further assumed to be extension-closed and of full characteristic, then we claim that 
\begin{equation}
\label{Eq:c0}
c_0 = \left( (m_0-1)!\right)^\frac{1}{m_0-2}, \mbox{ where } m_0 \geq 5.
\end{equation}
First, we note that the sequence $(a_r)_{r\geq1}$ given by 
\[
a_r \coloneqq \begin{cases} 1,& r = 1, 2\\
((r-1)!)^{1/(r-2)},& r\geq 3 \end{cases}
\]  
is weakly increasing for all $r$, and satisfies $a_r < a_{r+1}$ for all $r\geq 2$. Indeed, taking logarithms, this  inequality is seen to be equivalent, for $r\geq 3$, to the assertion that $(r-1)! < r^{r-2}$, the latter inequality being verified in this range by a trivial induction.  
Now, by the minimal choice of $m_0$ we have $A_{m_0-1} \in \mathfrak{Y}$, and 
since $\mathfrak{Y}$ is of full characteristic, we also have $C_2 \in \mathfrak{Y}$. 
Therefore, $\textsc{e}$-closure of $\mathfrak{Y}$ yields 
$S_{m_0-1} = A_{m_0-1} \rtimes C_2 \in \mathfrak{Y}$. From this we conclude that
$(\dagger)$ holds with $c_0 = ((m_0-1)!)^\frac{1}{m_0-2}$ for all $n<m_0$, and it remains
to examine the case $n=m_0$, where $m_0\geq 5$ by a combination of Lemmas~\ref{Lem:XProps}(iii), \ref{Lem:HypTranslate}, and \ref{Lem:brycecossey}. Let $G$ be a $\mathfrak{Y}$-subgroup of $S_{m_0}$,
and note that $G < S_{m_0}$, for if $G=S_{m_0}$ then the subgroup-closure of 
$\mathfrak{Y}$ yields $A_{m_0} \in \mathfrak{Y}$, contradicting the definition of $m_0$. 

Also, $G \neq A_{m_0}$ since $A_{m_0} \notin \mathfrak{Y}$, and thus $(S_{m_0}:G) \geq m_0$
(with equality if and only if $G \cong S_{m_0-1}$); cf. \cite[Chap.~II, Satz~5.3a)]{Huppert} noting that $m_0 \geq 5$, since $\mathfrak{N} \subseteq \X$ by Lemma \ref{Lem:brycecossey} and thus $\mathfrak{S} \subseteq \overline{\X}$. Hence, 
\[
|G| \leq (m_0-1)! < ((m_0-1)!)^\frac{m_0-1}{m_0-2} = c_0^{m_0-1},
\]
whence \eqref{Eq:c0}.

Since for our purposes will shall require $\mathfrak{Y}$ to be $\textsc{e}$-closed and of full characteristic, we therefore  define 
\[c_0 \coloneqq \left( (m_0-1)!\right)^\frac{1}{m_0-2},\] 
where $m_0 \geq 5$ is the least positive integer such that $A_{m_0} \notin \mathfrak{Y}$. Next, choose the natural number $n_0$ minimal such that $n_0 \geq m_0$, and such that $(\dagger)$ holds with $c = c_0$ whenever $G$ is a primitive $\Y$-group and $n \geq n_0$. The possibility for making this choice is argued in the proof of 
\cite[Thm.~5.8A]{DM}; cf. \cite[Thm.~5.6B]{DM}. 

Recent results of Mar\'oti \cite{Mar} allow us to derive numerical estimates for the invariant $n_0$ in terms of $m_0$ alone. Mar\'oti shows in particular that, if $G$ is primitive of degree $n\geq 25$, then $\vert G\vert < 2^n$, and that if $G$ is primitive with $G\neq A_n, S_n$, then $\vert G\vert <3^n$ for all $n$; cf. \cite[Cor. 1.2]{Mar}. These results represent a substantial improvement over the Wielandt-Praeger-Saxl bound $\vert G\vert \leq 4^n$ for $G\neq S_n$ and all $n$; see \cite{PraeSax}. Combining these estimates with Part (vii) of Prop.~\ref{Prop:EClosProps}, it is straightforward to obtain the following useful result.

\begin{lemma}\label{Lem:m_0=6}
Let $\X$ be a subgroup-closed class of finite groups, not containing all finite groups, and let $m_0$, $c_0$, and $n_0$ be as defined above. 
\begin{enumerate}[label={\upshape(\roman*)}]
\item If $m_0=6$ and $A_7\not\in\overline{\X},$ then $6\leq n_0\leq 13$.
\item If $7\leq m_0\leq 24$ and $A_{m_0+1}\not\in \overline{\X},$ then $m_0 \leq n_0\leq m_0+2$.
\item If $m_0\geq 25,$ then we have $n_0=m_0$. 
\end{enumerate} 
\end{lemma}

Finally, choose $c$ minimal such that $c \geq c_0$ and $(\dagger)$ holds for all $n \leq n_0$. The fact that there is a minimal choice for $c$ can be seen by considering ``$=$" in $(\dagger)$ instead of ``$\leq$" and taking $|G|^{\frac{1}{n-1}}$ maximal. Explicitly, we have 
\[
c = \max \left\{ |G|^{\frac{1}{n-1}} : \mathfrak{Y} \ni G \leq S_n, \, 2 \leq n \leq n_0\right\}.
\]
Note, here, that if $G_0$ is such that $c = \lvert G_0 \rvert^{\frac{1}{m-1}}$ for some $m = m(G_0) \leq n$, then $m$ is the minimal permutation representation degree of $G_0$. This is to say that $G_0$ cannot be embedded as a subgroup of any $S_r$, for $r < m$; otherwise, the maximality of $c = \lvert G_0 \rvert^{\frac{1}{m-1}}$ would be violated.
The constant $c$ thus produced is a function of $\mathfrak{Y}, m, c_0,$ and $n_0$ where the role of each variable is as we have explained. Let us then denote this dependence by 
\begin{equation}
c = c(\mathfrak{Y}) = \beta_{\mathfrak{Y}} = \beta(\mathfrak{Y},m,c_0,n_0).\footnote{Of course, the variables $m, c_0$, and $n_0$ all depend on $\mathfrak{Y}$.}
\end{equation}
\begin{remark}
The content of \cite[Thm.~5.8A]{DM} is that $(\dagger)$ holds with $c=\beta_{\mathfrak{Y}}$ for all $n \in \mathbb{N}$.
\end{remark}

Finally, let us define one more constant $\gamma = \gamma(\mathfrak{Y}, \beta_{\mathfrak{Y}})$, with $\mathfrak{Y}$ a class of finite groups having the properties previously stated. Write $\mathfrak{J}^{\ast}$ for the class of non-abelian finite simple groups, and set $\gamma \coloneqq \frac{\lambda}{\lambda + 1}$, with 
\begin{equation}
\lambda = \lambda_{\mathfrak{Y}} \coloneqq \min\left\{\frac{\log\left\lvert S \right\rvert}{\log(\beta_{\mathfrak{Y}}\left\lvert\Out(S)\right\rvert)} : S \in \mathfrak{J}^{\ast}\setminus \mathfrak{Y} \right\},
\end{equation}
so that 
\begin{equation}
\gamma = \gamma(\mathfrak{Y}) = \frac{\log \left\lvert S_0\right\rvert}{\log \left(\beta_{\mathfrak{Y}}\left\lvert \Aut(S_0)\right\rvert\right)},
\end{equation}
where $S_0 \in \mathfrak{J}^{\ast}\setminus \mathfrak{Y}$ is such that the minimum value $\lambda$ is attained. 

\begin{definition}
Given an image-closed, subgroup-closed, and extension-closed class of finite groups $\mathfrak{Y}$ of full characteristic, which is non-empty and does not contain every finite group, we shall refer to the uniquely defined constants $\beta_{\mathfrak{Y}}$ and $\gamma_{\mathfrak{Y}}$
 relative to $\mathfrak{Y}$ as \emph{$\mathfrak{Y}$-induced}.
\end{definition}

We can now state and prove our main result.

\begin{theorem}\label{Thm:Main}
Let $\X \subset  \mathfrak{E}$ be a non-empty subgroup-closed Fitting formation of full characteristic, and let $\overline{\X}$ be the \textsc{e}-closure of $\X$. Then the inequality
\begin{equation}
\label{Eq:MainThmIneq}
\left\lvert G^{\overline{\X}}\right\rvert \,>\, \vert G\vert^\gamma
\end{equation}
holds for all finite groups $G$ with $G\neq  1$ and $G_{\X}= 1$, where $\gamma = \gamma_{\overline{\X}}$ and the implied constant $\beta = \beta_{\overline{\X}}$ are $\overline{\X}$-induced. Moreover, the exponent $\gamma$ 
% in inequality 
% \eqref{Eq:MainThmIneq}
is best possible.
\end{theorem}

% \begin{remark}
% {\bf Give more direct definition of $\beta_{\overline{\X}}$, comment on optimality of the exponent $\gamma$, and the necessity of the radical condition in the theorem.}
% \end{remark}

% [Proof of Theorem~\ref{Thm:Main}]

\begin{proof}
We use induction on $\vert G\vert$ to prove inequality \eqref{Eq:MainThmIneq}, leaving the proof of optimality of the exponent $\gamma$ for the next section. As our main assertion is in the form of an implication, we may start the induction with $G= 1$, the hypothesis being false in that case. Let $G$ be a finite group such that $\vert G\vert = m>1$ and $G_{\X} =  1$, and suppose that our claim holds for all groups of order less than $m$. Let $N$ be a minimal normal subgroup of $G$, and set
\[
(G/N)_{\overline{\X}}\, =\, N_1/N
\]
where $N\leq N_1\unlhd G$. We now distinguish two cases.

\textbf{Case~1.} $N_1 < G$. By Part~(v) of Lemma~\ref{Lem:XProps} applied to the class $\overline{\X}$ and the group $G/N$, we have
\[
(G/N_1)_{\overline{\X}} \,=\,  1,
\]
as $\overline{\X}$ is both normal-product-closed and extension-closed. Since
\[
(G/N_1)_{\X} \,\leq (G/N_1)_{\overline{\X}}
\]
by Part~(iv) of Lemma~\ref{Lem:RadProps}, we find that $(G/N_1)_{\X} =  1$. By the case assumption, $G/N_1 > 1$ and, since $N_1\geq N > 1$, we may apply the induction hypothesis to the group $G/N_1$, to obtain
\begin{equation}\label{Eq:CaseI1}
\left(G^{\overline{\X}} : G^{\overline{\X}} \cap N_1\right) = 
\left\lvert G^{\overline{\X}} N_1/N_1\right\rvert = 
\left\lvert (G/N_1)^{\overline{\X}}\right\rvert > 
\left\lvert G/N_1\right\rvert ^\gamma,
\end{equation}
where we have used Part~(i) of Lemma~\ref{Lem:Res}, applied to the class $\overline{\X}$, the group $G$, and the normal subgroup $N_1$ of $G$, for the second equality. 
Furthermore, we have
\[
(N_1)_{\X}\,\leq\, G_{\X} \,=\,  1,
\]
so that $(N_1)_{\X} =  1$; indeed, from the fact that $(N_1)_{\X}$ is characteristic in $N_1$ 
by Part~(i) of Lemma~\ref{Lem:RadProps}, with $N_1$ being normal in $G$, we deduce that $(N_1)_{\X}$ 
is a normal $\X$-subgroup of $G$, so contained in the $\X$-radical of $G$; alternatively, to see this, one 
may observe that $(N_1)_{\X}$ is a subnormal $\X$-subgroup of $G$ by definition, and apply Part~(iii) of 
Lemma~\ref{Lem:RadProps}. Also, since $G^{\overline{\X}} \cap N_1 \unlhd G$, a fortiori 
$G^{\overline{\X}} \cap N_1 \unlhd N_1$, and we have 
\[
N_1\big/(N_1 \cap G^{\overline{\X}}) \,\cong\, N_1 G^{\overline{\X}}\big/G^{\overline{\X}}\, \leq \,G\big/G^{\overline{\X}}\,\in\, \overline{\X},
\]
so $N_1\big/(N_1 \cap G^{\overline{\X}}) \in\overline{\X}$, as $\overline{\X}$ is subgroup-closed and isomorphism-closed. It follows that 
\[
G^{\overline{\X}} \cap N_1 \,\geq\, (N_1)^{\overline{\X}}.
\]
Applying the induction hypothesis to the group $N_1 < G$, we now find that
\begin{equation}
\label{Eq:CaseI2}
\left\lvert G^{\overline{\X}} \cap N_1\right\rvert \,\geq \left\lvert(N_1)^{\overline{\X}}\right\rvert \,>\, \vert N_1\vert^\gamma.
\end{equation}
Combining \eqref{Eq:CaseI1} and \eqref{Eq:CaseI2} yields
\[
\left\lvert G^{\overline{\X}}\right\rvert \,=\, 
\left\lvert G^{\overline{\X}} \cap N_1\right\rvert\cdot \left(G^{\overline{\X}} : G^{\overline{\X}} \cap N_1\right) \,>\, 
\left\lvert N_1\right\rvert^\gamma\cdot\left\lvert G/N_1\right\rvert^\gamma \,=\, 
\left\lvert G\right\rvert^\gamma,
\]
which is our claim for the group $G$.

\textbf{Case~2.} $N_1 = G$. Here, $G/N = N_1/N \in\overline{\X}$, so, by Parts~(i)-(ii) of Lemma~\ref{Lem:Res},
\[
(G/N)^{\overline{\X}}\, =\,  1 \,= \,G^{\overline{\X}} N/N,
\]
implying that $G^{\overline{\X}} \leq N$. Since $G^{\overline{\X}} \unlhd G$ and $N$ is a minimal normal subgroup of $G$, the only possibilities are $G^{\overline{\X}} =  1$ and $G^{\overline{\X}} = N$. In the first case, we have $G\in \overline{\X}$ and $G_{\X} =  1$, implying $G =  1$ by Part~(ii) of Proposition~\ref{Prop:EClosProps}, a contradiction. Hence, since $\mathfrak{A} \subseteq \X$ (by Lemma~\ref{Lem:brycecossey}) and $G_{\X} =  1$, we have
\[
G^{\overline{\X}} = N = S^n = \underbrace{S\times \cdots \times S}_{n\,\mathrm{copies}},
\]
where $n$ is a positive integer and $S$ is a non-abelian simple group (if $S$ were abelian, then we would have $N\in\X$, thus $1<N\leq G_{\X}$, a contradiction). We also note that $S\not\in\overline{\X}$, since otherwise (by Part~(ix) of Proposition~\ref{Prop:EClosProps}), $S\in \X$, and we have $S \unlhd N \unlhd G$, implying $1 < S \leq G_{\X}$ by Part~(iii) of Lemma~\ref{Lem:RadProps}, again contradicting our hypothesis that $G_{\X} = 1$.

We claim that, in this second case, $N$ is the unique minimal normal subgroup of $G$. 
Indeed, suppose that $M\neq N$ is a second minimal normal subgroup of $G$. Then 
$N\cap M =  1$, so that $M$ embeds into $G/N$:
\[
M \,\cong\, M\big/(N\cap M)\,\cong \,NM/N \,\leq\, G/N,
\]
thus $M\in\overline{\X}$, since $G/N\in\overline{\X}$ and $\overline{\X}$ is subgroup-closed. 
Since $M\neq 1$, it follows by Part~(iv) of Proposition~\ref{Prop:EClosProps} 
and Part~(i) of Lemma~\ref{Lem:RadProps} that $M_{\X}$ is a non-trivial normal 
$\X$-subgroup of $G$, contradicting the fact that $G_{\X} =  1$. 

The fact that $G$ has a unique minimal normal subgroup now implies that $\cent_G(N) =  1$. 
Indeed, since $N\unlhd G$, the centraliser $\cent_G(N)$ is normal in $G$; 
hence, if $\cent_G(N) \neq  1$, we would have $N\leq \cent_G(N)$, ensuring that $N$ is abelian. 
This, in turn would force $G_{\X}$ to be non-trivial, since $\mathfrak{A} \subseteq \X$, a contradiction. 
It follows that the map $G\rightarrow \Aut(N)$ sending an element $g\in G$ to the restriction $\iota_g\vert_N$ 
of the inner automorphism $\iota_g$ of $G$, $\iota_g(x) = g^{-1} x g$, is an injective homomorphism, 
embedding $G$ in $\Aut(N) \cong \Aut(S)\wr S_n$, the wreath product being formed with respect 
to the natural action of $S_n$ on the standard $n$-set; cf., for instance, \cite[Lem.~9.24]{Rose} for the assertion on the structure of $\Aut(N)$.  
Consequently, $G/N \in\overline{\X}$ is embedded in the group $\Out(N) \cong \Out(S) \wr S_n$. 
Let $H$ denote the image of $G/N$ in $S_n$ under the natural homomorphism $\Out(S) \wr S_n \rightarrow S_n$. 
Since $\overline{\X}$ is image-closed, $H$ is an $\overline{\X}$-subgroup of $S_n$. 
Applying \cite[Thm.~5.8A]{DM} along with our preliminary observations, we get $\vert H\vert \leq \beta^{n-1}$, 
where $\beta = \beta_{\overline{\X}} > 1$ is $\overline{\X}$-induced. Thus,
\[
\lvert G/N\rvert \,\leq\, \left\lvert \Out(S)\right\rvert^n \cdot \vert H\vert \,\leq\, \left\lvert\Out(S)\right\rvert^n \cdot \beta^{n-1}.
\]
As $G^{\overline{\X}} = N$, and we want to show that $\vert N\vert >\vert G\vert^\gamma$, it suffices to prove that
\[
\lvert G/N\rvert \,<\, \vert N\vert^{\frac{1-\gamma}{\gamma}} \,=\, \vert S\vert^{\frac{n(1-\gamma)}{\gamma}}.
\]
For that, in turn, it is enough to establish that 
\[
\left(\beta_{\overline{\X}}^{\frac{n-1}{n}} \left\lvert\Out(S)\right\rvert\right)^{\frac{\gamma}{1-\gamma}} \,<\, \vert S\vert
\]
holds for each positive integer $n$ and every finite non-abelian simple group $S \in \mathfrak{J}^{\ast}\setminus \overline{\X}$. 
Since the sequence $\frac{n-1}{n}$ is strictly increasing and converges to unity as $n$ tends to infinity, 
it suffices to establish the bound 
\begin{equation}
\label{Eq:CritIneq}
\big(\beta_{\overline{\X}} \left\lvert \Out(S)\right\rvert\big)^{\frac{\gamma}{1-\gamma}} \,\leq\, \vert S\vert,
\end{equation}
for all finite non-abelian simple groups $S \notin \overline{\X}$. 
% We also note that, since $\mathfrak{A} \subseteq \X$, we have $\beta_{\overline{\X}} \geq (24)^{1/3}$; cf.~\cite[Thm.~5.8B]{DM}. 

Clearly, ~\eqref{Eq:CritIneq} is equivalent to 
\begin{equation}\label{Eq:CritIneq'}
\frac{\gamma}{1-\gamma} \leq \frac{\log\left\lvert S \right\rvert}{\log(\beta_{\overline{\X}}\left\lvert\Out(S)\right\rvert)}
\end{equation}
being valid for all non-abelian finite simple groups $S \notin \overline{\X}$, thus we require that
\begin{equation}\label{Eq:Critgamma}
\gamma \leq \frac{\lambda}{1+\lambda},\,\,\,\, \text{where }\lambda \coloneqq \min\left\{\frac{\log\left\lvert S \right\rvert}{\log(\beta_{\overline{\X}}\left\lvert\Out(S)\right\rvert)} : S \in \mathfrak{J}^{\ast}\setminus \overline{\X} \right\}.
\end{equation}
This is valid, however, by the very definition of $\gamma$. Our proof of (\ref{Eq:MainThmIneq}) is thus complete.
\end{proof}

\section{Optimality of the exponent in Theorem~\ref{Thm:Main}}\label{Sec:Optimality}
Having proved inequality \eqref{Eq:MainThmIneq}, it remains to establish sharpness, i.e., the fact that the exponent $\gamma$ in \eqref{Eq:MainThmIneq} is chosen 
as large as possible to make the concluding inequality valid for all finite groups $G$ with $G\neq 1$ and $G_{\X}=1$. 

Let $S_0 \in \mathfrak{J}^{\ast}\setminus \overline{\X}$ be such that 
\begin{equation}
\frac{\log\left\lvert S_0 \right\rvert}{\log(\beta_{\overline{\X}}\left\lvert\Out(S_0)\right\rvert)} =
\lambda = \min\left\{\frac{\log\left\lvert S \right\rvert}{\log(\beta_{\overline{\X}}\left\lvert\Out(S)\right\rvert)} : S \in \mathfrak{J}^{\ast}\setminus \overline{\X} \right\}.
\end{equation}
We have defined the $\overline{\X}$-induced $\beta$ to be
\[
\beta = \max \left\{ |G|^{\frac{1}{n-1}} : \overline{\X} \ni G \leq S_n,\, 2 \leq n \leq n_0\right\},
\]
so let $L$ be an $\overline{\X}$-subgroup of some $S_\nu$, $2\leq \nu \leq n_0$, for which the maximum in the definition of $\beta$ is attained. 
Let $r = r_n = {\nu}^n$, and define the sequence of
groups $L_1 \coloneqq L$, $L_r \coloneqq L_{r_{n-1}} \wr L$ for $n \geq 2$.
An easy induction shows that $L_r \in \overline{\X}$ for all $r = r_n$, since $\overline{\X}$ is $\textsc{d}_0$-closed and $\textsc{e}$-closed. Moreover, $\lvert L_r \rvert = |L|^\frac{r-1}{\nu-1}$. Now, define the sequence of groups 
\[
W_r \coloneqq \Aut(S_0) \wr L_r.
\]
It is not difficult to see that, for all $r$, $\left( W_r\right)_{\X} = 1$, and
\[
(W_r)^{\overline{\X}} \cong S_0^r = \underbrace{S_0\times \cdots \times S_0}_{r\,\mathrm{copies}}. 
\]
We take a moment to justify these claims. Recall the content of Goursat's lemma~\cite{og}, which provides a method for finding (normal) subgroups of direct products. An account (and some generalisations) of this useful result are given in~\cite{goursat}.
Briefly, given subgroups $Q \lhd R \leq G_1$, $S \lhd T \leq G_2$ and an isomorphism 
$f : R/Q \to T/S$, 
\[
H = \left\{(a, b) \in R \times T : f(aQ) = bS\right\}
\]
is a subgroup of $G_1 \times G_2$ 
and each subgroup of $G_1 \times G_2$ is of this form. 
A necessary and sufficient condition for $H$ to be normal in $G_1 \times G_2$ is that both $Q, R$ and
$S, T$ are normal in $G_1$, $G_2$ respectively, and that $Q/R \leq \Zent(G_1/R)$, $T/S \leq \Zent(G_2/S)$.
An immediate consequence of Goursat's lemma is that if $G$ is semisimple, i.e., a direct product of non-abelian simple groups, then every (non-trivial) normal subgroup of $G$ is again semisimple, a subproduct of the direct factors of $G$,
and thus the same holds true for every quotient of $G$.

Now, the subgroup $S_0^r$ is characteristic in the base group $\Aut(S_0)^r$, being its socle, and the base group is normal in $W_r$, so $S_0^r$ is normal in $W_r$. On the other hand,
\[
W_r\big/S_0^r \cong \Out(S_0) \wr L_r,
\]
where $\Out(S_0)$ is soluble by Schreier's Conjecture (now accepted as a Theorem, owing to CFSG). Since $\Out(S_0), L_r \in \overline{\X}$ for all $r$, it follows by  $\textsc{e}$-closure and $\textsc{d}_0$-closure of $\overline{\X}$ that 
$W_r\big/S_0^r \in \overline{\X}$. Thus $(W_r)^{\overline{\X}}$ is a normal subgroup of $W_r$, contained in $S_0^r$. From the consequence to Goursat's lemma, we deduce that 
$S_0^r \Big/ (W_r)^{\overline{\X}}$ is either the trivial group, or isomorphic to the direct product of copies of $S_0$.
We discount the latter possibility by noting that 
\[
S_0^r \Big/ (W_r)^{\overline{\X}} \lhd W_r \Big/ (W_r)^{\overline{\X}} \in \overline{\X},
\]
whilst $S_0 \notin \overline{\X}$. This settles our claim that the $\overline{\X}$-residual of $W_r$ is the socle of its base group.

It remains to justify the claim that the $\X$-radical of $W_r$ is trivial. For that it will suffice to establish that a non-trivial normal subgroup of $W_r$ contains $S_0$ as a section. (Since $S_0 \notin \overline{\X}$, it follows that $S_0 \notin \X$, while the $\textsc{s}\textsc{q}$-closure of $\X$ forces sections of $\X$-groups to be $\X$-groups.)

Let $K$ denote the base of $W_r$, so that $\Soc(K) = S_0^r$. If $1 \neq N \lhd W_r$ and $N \cap K > 1$, then $N \cap \Soc(K) > 1$, 
and thus $N \cap \Soc(K)$ is a normal subgroup of $W_r$ and isomorphic to $S_0^\rho$ for some $\rho \leq r$, by our previous discussion. 
We may thus assume that $N \cap K = 1$. Next, we argue that we can also assume that $N$ projects onto $L_r$, the top group of $W_r$, 
under the natural homomorphism $W_r \to W_r/K$. For, results of Neumann (cf. \cite[Lem.~8.1]{neumann} and the discussion preceding it) 
show that $NK$ is again a wreath product ``of the same type'' as $W_r$; and, of course, in that case $N$ projects \emph{onto} its image in $NK$.
Hence we can choose $N$ minimal subject to $N \cap \Soc(K) = 1$, $\overline{N} < \overline{W_r} \cong L_r$ and $S_0$ not involved in $N$, 
and derive a contradiction. Now, Neumann goes on to show that, under the hypotheses $1 \neq N \lhd W_r$ and $NK = W_r$, $N$ contains 
the normal subgroup $M=[K,L_r]$ of $K$; cf. \cite[Lem.~8.2]{neumann} and \cite[Thm.~4.1]{neumann}. We argue that $M>1$. If $M$ were to be trivial, then $K$ would 
centralise $L_r$, and thus $L_r$ would be normal in $W_r$, so $W_r$ would directly decompose as $W_r = K \times L_r$, against the requirements
of \cite[Thm. 6.1]{neumann} ($K$ would need to be central in $W_r$, which in turn would force $S_0$ to be abelian.)
This establishes what we want, since we have assumed that $N \cap K =1$.

Further, if $\gamma_r$ is such that $\left \lvert (W_r)^{\overline{\X}} \right \rvert = \lvert W_r \rvert^{\gamma_r}$, then
\[
\gamma_r = 
% \frac{\log \left\lvert (G_r)^{\overline{\X}} \right\rvert}{\log \lvert G_r \rvert} 
\frac{\log \left(\lvert S_0 \rvert^r\right)}{\log \left(\lvert\Aut(S_0)\rvert^r\cdot|L|^\frac{r-1}{\nu-1}\right)}
=
\frac{\log \lvert S_0 \rvert}{\log \left(\lvert\Aut(S_0)\rvert\cdot|L|^\frac{r-1}{r(\nu-1)}\right)}.
\]
Since 
\[
\lim_{r \to \infty} |L|^\frac{r-1}{r(\nu-1)} = |L|^\frac{1}{\nu-1} = \beta, 
\]
we get $\lim_{r \to \infty} \gamma_r = \gamma$ which establishes optimality of $\gamma$ and sharpness of the inequality in our main result.
% Let $m$ be the least positive integer such that $A_{m+1} \notin \overline{\X}$. The existence of $m$ is guaranteed by our requirement that $\overline{\X} \neq \mathfrak{E}$, and the observation that every finite group can be embedded as a subgroup of some symmetric group, thus as a subgroup of some alternating group.  
% Let $H \coloneqq A_{m+1}$, $L \coloneqq S_m$. First, we argue that $L \in \overline{\X}$. For $L$ is a split extension $L = A_m \rtimes C_2$, with $C_2 \in \X \subseteq \overline{\X}$ and $A_m \in \overline{\X}$ by the minimal choice of $m$. The $\textsc{e}$-closure of $\overline{\X}$ now yields what we want. Define a sequence of groups by $L_1 \coloneqq L$ and inductively $L_r \coloneqq L_{r-1} \wr L$ for $r \geq 2$. Again by induction, $L_r \in \overline{\X}$ for all $r \in \mathbb{N}$. By $\textsc{d}_0$-closure and $\textsc{e}$-closure of $\overline{\X}$, we see that $L_r$ is an $\overline{\X}$-subgroup of $S_{m^r}$ of order $\left( m!\right)^{\frac{m^r-1}{m-1}}$.
\section{Some applications}\label{Sec:Appl}
Having established our main result, we easily deduce the original Theorem~\ref{Thm:SolRes} by taking $\X$ to be the class of all finite nilpotent groups. If $\X = \mathfrak{N}$, then clearly $\overline{\X} = \mathfrak{S}$ (every soluble group has a nilpotent subnormal series). Further, the least positive integer $m$ for which $A_m \notin \overline{\X}$ is $m=5$, while $n_0 = m = 5$ follows from \cite[Thm.~5.8B]{DM}. Thus 
\[
\beta_{\mathfrak{S}} = ((5-1)!)^{\frac{1}{5-2}} = 24^{\frac{1}{3}},
\]
and $\gamma_{\mathfrak{S}} = \frac{\lambda_{\mathfrak{S}}}{1+\lambda_{\mathfrak{S}}}$, where 
\begin{equation}
\lambda_{\mathfrak{S}} = \min\left\{\frac{\log\left\lvert S \right\rvert}{\log\left(24^{\frac{1}{3}}\left\lvert\Out(S)\right\rvert\right)} : S \in \mathfrak{J}^{\ast} \right\}.
\end{equation}
At this point we are met with a certain difficulty, which is to find the actual value of $\lambda_{\mathfrak{S}}$. Below, we offer a short proof that the minimum is attained precisely for $S=A_5$ and therefore 
 \[
\lambda_{\mathfrak{S}} = \frac{\log 60}{\log \left( 2 \cdot 24^{\frac{1}{3}} \right)} \approx 2.33629\,,
\]
which in turn yields
\[
\gamma_{\mathfrak{S}} = \frac{\log 60}{\log \left( 120 \cdot 24^{\frac{1}{3}} \right)} \approx  0.700265861,
\]
in agreement with the findings in \cite{DHKL}.

By \cite{SK}, we have $\lvert \Out (S) \rvert < \log_2|S|$ for all $S \in \mathfrak{J}^{\ast}$, and thus it suffices to prove that 
\begin{equation}\label{Eq:234}
\frac{\log\left\lvert S \right\rvert}{\log\left(24^{\frac{1}{3}} \cdot \log_2|S|\right)} > 2.34
\end{equation}
for all $S \in \mathfrak{J}^{\ast}\setminus \{A_5\}$. It is easy to check that \eqref{Eq:234} is valid for all $|S| \geq 3961$. There are 8 non-abelian simple groups of order $\in (60,3960]$, and of those, all but $A_6$ have outer automorphism groups of order 2, as does $A_5$; cf. \cite[p. 239]{atlas}. Thus the only rival simple group is $A_6$, and we check directly that 
\[
2.33629 \approx \frac{\log 60}{\log \left( 2 \cdot 24^{\frac{1}{3}} \right)} < \frac{\log \lvert A_6 \rvert}{\log \left( \lvert \Out(A_6)\rvert \cdot 24^{\frac{1}{3}} \right)} = \frac{\log 360}{\log \left( 4 \cdot 24^{\frac{1}{3}} \right)} \approx 2.40677 \, ,
\]
which proves our assertion.

% In the soluble universe--that is to say, if $\X$ is such that $\overline{\X}\subseteq \mathfrak{S}$--we argue that Theorem~\ref{Thm:SolRes} is best possible, so far as both the original proof of Theorem~\ref{Thm:SolRes} and our proof of Theorem~\ref{Thm:Main} require\footnote{This requirement is, of course, only implicit in the former case.} that $\X$ be of full characteristic; cf. the discussion leading to subsection \ref{prolegomena}.
% We note here that the sharpness of Theorem~\ref{Thm:SolRes} that we assert is not relative to the inequality in the theorem. Rather, we want to justify why  the condition $\fitt(G) = 1$ cannot be replaced by $\mathbf{Y}(G) = 1$, for some characteristic subgroup $\mathbf{Y}$ which satisfies $\mathbf{Y}(G) \leq \fitt(G)$ for all groups $G$, and thus we expand on \cite[Rem. 1]{DHKL}.

This raises the question if our Theorem~\ref{Thm:Main} is simply a restatement of Theorem~\ref{Thm:SolRes} using the more abstract language and theory of classes of groups. It is not so. Recall the seminal work of Thompson \cite{ngroups}
which classifies (among other things) the minimal simple groups, i.e., the (finite) non-abelian simple groups all of whose proper subgroups are soluble. These are:

\begin{itemize}
\item $\PSL_2(2^p)$, $p$ a prime;
\item $\PSL_2(3^p)$, $p$ an odd prime;
\item $\PSL_2(p)$, $p>3$ a prime congruent to 2 or 3$\mod 5$;
\item $\Sz(2^p)$, $p$ an odd prime;
\item $\PSL_3(3)$.
\end{itemize}

% It is clear that if $\X$ (or $\overline{\X}$), supposed to be a subgroup-closed Fitting formation, is to contain a non-abelian simple group, then subgroup-closure of $\X$ forces the said simple group to be one of those in the above list. 

Now, let $J$ be any minimal simple group in the list above. We let $\X \coloneqq \textsc{d}_0(J) \times \mathfrak{S}$ be the class of groups $G$ of the form 
\[
G = J_1 \times \cdots \times J_n \times H
\]
for $n \in \mathbb{N}$ with $J_i \cong J$, $i=1,\dots,n$ and $H \in \mathfrak{S}$. Then $\X$ is a subgroup-closed Fitting formation of full characteristic; for a proof see Example 1.6 in \cite[Chap. XI]{DH}. Further, $\overline{\X}$ is the class of groups having every non-cyclic composition factor isomorphic to $J$, and it is clear that $\mathfrak{S} \subset  \overline{\X} \subset \mathfrak{E}$. 

In fact, if $\mathfrak{F}$ is a subgroup-closed Fitting formation such that $\mathfrak{S} \subset  \overline{\mathfrak{F}} \subset \mathfrak{E}$, then $\overline{\mathfrak{F}}$ contains a subgroup-closed Fitting formation of the type described above. Since $\mathfrak{S} \subset \overline{\mathfrak{F}}$, it follows that $\overline{\mathfrak{F}}$, thus also $\mathfrak{F}$, contains non-abelian finite simple groups. Let $J \in \mathfrak{F} \cap \mathfrak{J}^{\ast}$ have least possible order. Then $\textsc{d}_0(J)$ is the class generated by $J$ and is a Fitting formation; cf. Example 2.13 in \cite[Chap. II]{DH}. The minimal choice of $J$ and the fact that $\mathfrak{F}$ is a subgroup-closed Fitting formation (so that sections of $\mathfrak{F}$-groups are again $\mathfrak{F}$-groups) together imply that $J$ is minimal simple. Now, $\mathfrak{N} \subset \mathfrak{F}$ by Lemma~\ref{Lem:brycecossey}, thus $\mathfrak{S} \subset \overline{\mathfrak{F}}$, and so $\textsc{d}_0(J) \times \mathfrak{S} \subseteq \overline{\mathfrak{F}}$. That $\textsc{d}_0(J) \times \mathfrak{S}$ is a subgroup-closed Fitting formation follows from our previous discussion.

Next, we want to discuss the computation of the optimal constants $m_0(\overline{\X})$, $n_0(\overline{\X})$, $\beta_{\overline{\X}}$, and $\gamma_{\overline{\X}}$ in the case where 
$\X = \textsc{d}_0(A_5) \times \mathfrak{S}$.
This will lead to an optimal inequality of the same type as in Theorem \ref{Thm:SolRes}. Consider $A_5 \cong \PSL_2(2^2)$, which fits into Thompson's list of minimal simple groups and take $\X = \textsc{d}_0(A_5) \times \mathfrak{S}$, so that $\overline{\X}$ is the class of groups having every non-abelian composition factor isomorphic to $A_5$. According to our previous discussion, both $\X$ and $\overline{\X}$ are $\textsc{s}$-closed Fitting formations of full characteristic. Clearly, the smallest $m$ such that $A_m \notin \overline{\X}$ is $m_0 = 6$, so that $c_0 = 120^{\frac{1}{4}} \approx 3.30975$. According to Lemma \ref{Lem:m_0=6} we have $6 \leq n_0 \leq 13$, and we wish to pin down the exact value of $n_0$. In fact, we claim that $n_0=6$, and to prove that it will suffice to establish that there exist no primitive $\overline{\X}$-groups of degree $n$, such that $7 \leq n \leq 12$ with 
\[|G|>c_0^{n-1} = 120^{\frac{n-1}{4}}.\] 
A list of primitive groups of small degree is readily available, e.g. \cite[Table 9.62]{combdesigns}, and the number of groups that require checking is small enough that the task can be carried out by hand. Alternatively, a computer algebra system can be employed.
Therefore, 
\[\beta_{\overline{\X}} = c = \max \left\{ |G|^{\frac{1}{n-1}} : \mathfrak{\overline{\X}} \ni G \leq S_n, 2 \leq n \leq 6\right\}.
\]
It's easy to check that the maximum is attained for $S_5 \in \overline{\X}$, so that $\beta_{\overline{\X}} = 120^{\frac{1}{4}}$.
The next step is to find the value of $\gamma_{\overline{\X}}$.
For that we shall need to estimate the constant
\begin{equation}
\lambda_{\overline{\X}} = \min\left\{\frac{\log\left\lvert S \right\rvert}{\log\left(120^{\frac{1}{4}}\left\lvert\Out(S)\right\rvert\right)} : S \in \mathfrak{J}^{\ast}\setminus \left\{A_5\right\} \right\}.
\end{equation}

We stipulate that the non-abelian simple group for which the minimum value $\lambda_{\overline{\X}}$ is realised is $A_6$, and we shall need to make use of Kohl's inequality \cite{SK} again in order to prove that, i.e., that $\lvert \Out (S) \rvert < \log_2|S|$ for all $S \in \mathfrak{J}^{\ast}$. Note here that $\left\lvert A_6 \right\rvert = 360$, while $\left\lvert\Out(A_6)\right\rvert = 4$, so that in essence we are claiming that 
\[
\lambda_{\overline{\X}} = \frac{\log 360}{\log\left(4*120^{\frac{1}{4}}\right)} \approx 2.27864.
\]

Thus it suffices to prove that 
\begin{equation}\label{Eq:228}
\frac{\log\left\lvert S \right\rvert}{\log\left(120^{\frac{1}{4}} \cdot \log_2|S|\right)} > 2.28
\end{equation}
for all $S \in \mathfrak{J}^{\ast}\setminus \{A_5, A_6\}$. It is easy to check that \eqref{Eq:228} is valid for all $|S| \geq 4529$. There are 8 non-abelian simple groups of order $\in (168,4529]$, excluding $A_6$, and we check directly that each of those groups yields a numerical value for $\lambda$ that is strictly bigger than that produced by $A_6$. 

We conclude that 
\[
\gamma_{\overline{\X}} = \frac{\lambda_{\overline{\X}}}{1+\lambda_{\overline{\X}}} = \frac{\log 360}{\log \left(1440*120^{\frac{1}{4}}\right)} \approx 0.694995\, .
\]
\bibliographystyle{amsalpha}

\end{document}